\documentclass[11pt,a4paper,leqno]{amsart}
\usepackage{amssymb, amsmath, amscd}
\usepackage[all]{xy}
\usepackage{color}

\newtheorem{theorem}{Theorem}[section]

\newtheorem{lemma}[theorem]{Lemma}

\newtheorem{remark}[theorem]{Remark}
\newtheorem{corollary}[theorem]{Corollary}

\usepackage{soul}

\begin{document}
\title[Homogeneous Ricci almost solitons]
{Homogeneous Ricci almost solitons}
\author[Calvi\~{n}o-Louzao, Fern\'andez-L\'opez, Garc\'ia-R\'io, V\'{a}zquez-Lorenzo]{E. Calvi\~{n}o-Louzao, M. Fern\'andez-L\'opez, E. Garc\'ia-R\'io, 
R. V\'{a}zquez-Lorenzo}
\address{(E. C.-L.) Department of Mathematics , IES Ram\'{o}n Caama\~{n}o, Mux\'ia, Spain}
\address{(M. F.-L.) Department of Mathematics , IES Mar\'\i a Sarmiento, Viveiro, Spain}
\address{(E. G.-R.) Faculty of Mathematics,
University of Santiago de Compostela, 15782 Santiago de Compostela, Spain}
\address{(R. V.-L.) Department of Mathematics , IES de Ribadeo Dionisio Gamallo, Ribadeo, Spain}

\email{estebcl@edu.xunta.es, manufl@edu.xunta.es, eduardo.garcia.rio@usc.es,
ravazlor@edu.xunta.es}
\thanks{Supported by projects GRC2013-045, MTM2013-41335-P and EM2014/009 with FEDER funds (Spain).}
\date{}

\begin{abstract}
It is shown that a locally homogeneous proper Ricci almost soliton is either of constant sectional curvature or a product $\mathbb{R}\times N(c)$, where $N(c)$ is a space of constant curvature.
\end{abstract}

\maketitle

\section{Introduction}

A Riemannian manifold $(M,g)$ is said to be a \emph{Ricci soliton} if there exists a vector field $X$ on $M$ satisfying the equation 
\begin{equation}\label{eq:RicciSoliton}
\mathcal{L}_Xg+\rho=\lambda g
\end{equation}
where $\mathcal{L}$ denotes the Lie derivative, $\rho$ stands for the Ricci tensor 
 and $\lambda$ is a constant. 
Recently, a generalization of Equation \eqref{eq:RicciSoliton} has been considered in \cite{PRRS}, allowing $\lambda$ to be a smooth function $\lambda\in\mathcal{C}^\infty(M)$. $(M,g,X,\lambda)$ is said to be a \emph{Ricci almost soliton} if the vector field $X$ satisfies the Equation \eqref{eq:RicciSoliton} for some function $\lambda$. Since Ricci almost solitons contain Ricci solitons as a special case, we say that the  Ricci almost soliton is \emph{proper} if the soliton function $\lambda$ is non-constant. 
In the special case when the vector field $X$ is the gradient of a function $f$, Equation \eqref{eq:RicciSoliton} becomes 
$H_f+\rho=\lambda g$ (after rescaling of the potential function $f$), where $H_f$ denotes the Hessian tensor, and one refers to $(M,g,f,\lambda)$ as a \emph{gradient Ricci almost soliton}. 

Ricci almost solitons exhibit some similarities and striking differences when comparing with usual Ricci solitons. For instance no K\"ahler manifold admits proper gradient Ricci almost solitons \cite{Ma} in opposition to the Ricci soliton case. We just refer to \cite{BBR--2012, Barros-Ribeiro12, BCR-2013, CMMR, Gh, PRRS, WGX} for some generalizations of Ricci solitons and further references on Ricci almost solitons. 

Two basic problems in the study of Ricci almost solitons are to classify them under geometric conditions and to investigate whether a Ricci almost soliton is a gradient. An answer to the second question was given in \cite[Corollary 2]{BBR-2014}, where it is shown that any compact Ricci almost soliton with constant scalar curvature is a gradient one. The geometry of Ricci almost solitons was investigated in the gradient case in \cite{PRRS}, where   complete Einstein gradient Ricci almost solitons were classified.

Our main purpose in  this paper is to analyze the existence of  Ricci almost solitons under some symmetry conditions by showing the following 

\begin{theorem}
\label{lema:homras}
A locally homogeneous Riemannian manifold $(M,g)$ admits a proper Ricci almost soliton if and only if $(M,g)$ is either a space of constant sectional curvature or locally isometric to a product  $\mathbb{R}\times N(c)$, where $N(c)$ is a space of constant curvature.
\end{theorem}

The situation is even more restrictive when gradient Ricci almost solitons are considered. We show in 
Lemma \ref{lemma-irreducibility} that proper gradient Ricci almost solitons are locally irreducible, from where one immediately gets:

\begin{corollary}
\label{th:lochomogars}
A locally homogeneous Riemannian manifold $(M,g)$  admits a proper gradient Ricci almost soliton if and only if $(M,g)$ is a space of non-zero constant sectional curvature.
\end{corollary}

The above results are in sharp contrast with the Ricci soliton case. Recall that locally homogeneous gradient Ricci solitons are rigid \cite[Theorem 1.1]{Petersen-Wylie09}, i.e., they are isometric to a product manifold $M=N\times\mathbb{R}^k$, where $N$ is an Einstein manifold and the potential function of the soliton is given by the norm of the position vector in the Euclidean part. 
Also note that while there exist many non-gradient homogeneous Ricci solitons (cf. \cite{Jablonski, Lauret} and references therein), the existence of proper Ricci almost solitons is a very restrictive condition in the homogeneous setting.

\medskip

A more general condition than  local homogeneity is that of \emph{curvature homogeneity} (i.e.,  for any two points $p,q\in M$ there is a linear isometry $\psi_{pq}:T_pM\rightarrow T_qM$ preserving the curvature tensor, $\psi_{pq}^*R_q=R_p$). Any locally homogeneous Riemannian manifold is curvature homogeneous and conversely if $(M,g)$ is $k$-curvature homogeneous (i.e.,  for any two points $p,q\in M$ there is a linear isometry $\psi_{pq}:T_pM\rightarrow T_qM$ preserving the curvature tensor and its covariant derivatives up to order $k$, $\psi_{pq}^*\nabla^\ell R_q=\nabla^\ell R_p$, $\ell=0,\dots,k$) for sufficiently large $k$, then $(M,g)$ is locally homogeneous (see \cite{BKoVa, Gi} for more information). There are, however, curvature homogeneous Riemannian manifolds whose curvature tensor does not correspond to any homogeneous space. 
Curvature homogeneous gradient Ricci almost solitons are described in the following result, which generalizes Corollary \ref{th:lochomogars}.

\begin{theorem}\label{th:curvhomog}
A curvature homogeneous Riemannian manifold  $(M,g)$ is a proper gradient Ricci almost soliton if and only if $(M,g)$  is a space of non-zero constant sectional curvature.
\end{theorem}

Due to the irreducibility of proper gradient Ricci almost solitons established in Lemma \ref{lemma-irreducibility}, a gradient Ricci almost soliton is rigid if and only if it is Einstein. 
The existence of complete Einstein proper gradient  Ricci almost solitons which are not of constant sectional curvature (see \cite[Theorem 2.3]{PRRS}) evinces the existence of some inaccuracies in  \cite[Theorem 1]{BGR2013}, where it is claimed that any proper gradient  Ricci almost soliton whose Ricci tensor is Codazzi (i.e., its covariant derivative is totally symmetric $(\nabla_X\rho)(Y,Z)=(\nabla_Y\rho)(X,Z)$) is of constant curvature. Explicit examples of non-Einstein gradient Ricci almost solitons with Codazzi Ricci tensor can be constructed as follows. Let  $M=\{{\vec{x}=(x_1,\dots,x_n)}\in\mathbb{R}^n : x^n>0\}$ and, for $n\geq 3$, consider the metric  $g=x_n^\frac{4}{n-2}\left( dx_1^2+\cdots+dx_n^2\right)$. Then the Ricci tensor is Codazzi but not parallel (cf.  \cite{Gray}) and a straightforward  calculation shows that
$$
	{f(\vec{x})=}\frac{(n-2) x_n^{\frac{n+2}{n-2}}}{n+2}+\frac{2 n \log \left(x_n\right)}{n+2}
	\quad \text{and} \quad 
	{\lambda(\vec{x})=} \frac{2}{(n-2) x_n} { +} \frac{2 x_n^{-\frac{2 n}{n-2}}}{{n+2}}
$$ 
define a gradient  Ricci almost soliton on $(M,g)$.

The above shows that, in contrast with gradient Ricci solitons (cf. \cite{PW1}), gradient Ricci almost solitons whose Ricci tensor is Codazzi  are not necessarily rigid, although they have constant scalar curvature. Hence we investigate some rigidity conditions as follows:

\begin{theorem}\label{th:rigidity}
Let $(M,g,f,\lambda)$ be an $n$-dimensional gradient Ricci almost soliton with Codazzi Ricci tensor. 
\begin{enumerate}
\item[(i)] If $(M,g)$ is radially flat (i.e., $R(\,\cdot\,,\nabla f)\nabla f=0$), then it is a rigid gradient Ricci soliton.
\item[(ii)]  If $\nabla f$ is an eigenvector of the Ricci operator with $\operatorname{Ric}(\nabla f)=\frac{\tau}{n}\nabla f$ then $(M,g)$ is Einstein, where $\tau$ denotes the scalar curvature.
\end{enumerate}
\end{theorem}

Finally some rigidity criteria are given for gradient Ricci almost solitons in terms of the Ricci curvatures. Recall that a complete gradient Ricci soliton is rigid if and only if the Ricci curvatures take values $\rho_{i}\in\{ 0,\lambda\}$ \cite{FL-GR}. Also, a gradient shrinking (resp., expanding) Ricci soliton is rigid if and only if the scalar curvature is constant and the Ricci curvatures are bounded $0\leq\rho_i\leq\lambda$ (resp., $\lambda\leq\rho_i\leq 0$) \cite{PW1}. These results extend to the gradient Ricci almost soliton setting as follows:

\begin{theorem}
\label{th:2}
Let $(M,g,f,\lambda)$ be a gradient Ricci almost soliton. 
\begin{enumerate}
\item[(i)] If $M$ is compact and the Ricci curvatures  $\rho_{i}\in\{ 0,\lambda\}$, then $(M,g)$ is Einstein.
\item[(ii)] If the scalar curvature is constant and $0\leq\rho_i\leq\lambda$, then the soliton function $\lambda$ is constant provided that it attains a global minimum. 
\item[(iii)] If the scalar curvature is constant and  $\lambda\leq\rho_i\leq0$, then the soliton function $\lambda$ is constant provided that it attains a global maximum. 
\end{enumerate} 
\end{theorem}

\begin{remark}\label{re.nuevo}
\rm
The results in Theorem \ref{th:2} are still valid in the generic case of non-gradient Ricci almost solitons just assuming that they are either shrinking (i.e., the soliton function $\lambda>0$) or expanding ($\lambda<0$).
\end{remark}

\section{Proofs}
\label{se:Homogeneous}

\begin{proof}[Proof of Theorem \ref{lema:homras}]
Let $(M,g,X,\lambda)$ be an $n$-dimensional proper Ricci almost soliton. If $(M,g)$ is locally homogeneous, then  for each point $p\in M$ there exist a local basis of Killing vector fields $\{ \xi_1,\dots,\xi_n\}$. Taking the Lie derivative with respect to $\xi_i$ on the Ricci almost soliton equation,
one gets 
$$
\begin{array}{rcl}
d\lambda(\xi_i)g&=&\mathcal{L}_{\xi_i}(\mathcal{L}_Xg+\rho)
=\mathcal{L}_{\xi_i}\mathcal{L}_Xg\\
\noalign{\medskip}
&=&\mathcal{L}_X\mathcal{L}_{\xi_i}g+\mathcal{L}_{[\xi_i,X]}g
=\mathcal{L}_{[\xi_i,X]}g\,.
\end{array}
$$
Hence, $X_\xi=[\xi,X]$ is a conformal vector field  for each Killing vector field $\xi$ on $(M,g)$.

First of all, observe that if all vector fields $X_\xi$ vanish, then $d\lambda(\xi)=0$ for all Killing vector fields and thus $\lambda$ is constant due to the existence of local bases of Killing vector fields by local homogeneity. In this case the Ricci almost soliton becomes a Ricci soliton.
Hence we assume that there exists a Killing vector field $\xi$ on $(M,g)$ such that $X_\xi\neq 0$ is a non-Killing conformal vector field.

In the special case that $X_\xi=[\xi,X]$ is a homothetic vector field (i.e.,  $d\lambda(\xi)=\operatorname{const.}$),  fix a point $p\in M$ and use a basis of Killing vector fields $\{\xi_i\}$ to express $X_\xi(p)=\sum_{\ell}\kappa^\ell\xi_\ell(p)$. Then $Z=X_\xi-\sum_{\ell}\kappa^\ell\xi_\ell$ is a homothetic vector field which vanishes at $p\in M$. Since homothetic vector fields preserve the Ricci tensor ($\mathcal{L}_Z\rho=0$) and $Z(p)=0$, one has that the eigenvalues of the Ricci tensor vanish at $p\in M$. Hence $(M,g)$ is Ricci flat and thus flat by local homogeneity \cite{Spiro}.

Assume $X_\xi$ is a non-homothetic conformal vector field with divergence
$\operatorname{div}X_\xi=\frac{n}{2}d\lambda(\xi)$. 
Since $X_\xi$ is conformal, its local flow consists of conformal transformations (which preserve the  Weyl conformal curvature tensor of type $(1,3)$) and hence one has
$$
\mathcal{L}_{X_\xi}g=\frac{2}{n}\operatorname{div}(X_\xi)g\,,
\qquad
\mathcal{L}_{X_\xi}W=2\operatorname{div}(X_\xi)W\,,
$$
where $W$ denotes the Weyl conformal curvature tensor of type  $(0,4)$.
Hence, we have
$$
\begin{array}{rcl}
X_\xi \| W\|^2&=&X_\xi g(W,W)=(\mathcal{L}_{X_\xi}g)(W,W)+2 g(\mathcal{L}_{X_\xi}W,W)\\
\noalign{\medskip}
&=&d\lambda(\xi) g(W,W)+2nd\lambda(\xi) g(W,W) =(2n+1)\| W\|^2 d\lambda(\xi)\,.
\end{array}
$$
By local homogeneity, the norm of the Weyl conformal curvature tensor is constant on $M$, and thus $\| W\|^2=0$ since $d\lambda(\xi)\neq 0$. This shows that $(M,g)$ is locally conformally flat and hence locally symmetric \cite{Takagi}. Therefore $(M,g)$ is either of constant sectional curvature or a product $\mathbb{R}\times N(c)$ where $N(c)$ is a space of constant sectional curvature $c$, or a product $N_1(c)\times N_2(-c)$.

Finally observe that since any space of constant curvature $M(c)$ is Einstein, Ricci almost solitons are nothing but conformal vector fields, which shows that $M(c)$ is a Ricci almost soliton (see, for example \cite{Kanai83}) for any value of $c\in\mathbb{R}$. Moreover  recall that the products $\mathbb{R}\times N(c)$ are rigid gradient Ricci solitons just considering the potential function 
$f=\frac{\mu}{2} (\pi_\mathbb{R})^2$, where $\pi_\mathbb{R}$ denotes the projection on $\mathbb{R}$ and the constant $\mu$ is determined by the sectional curvature $\mu=(n-1)c$. Let $Z$ be a non-homothetic conformal vector field on $\mathbb{R}\times N(c)$ (i.e., $\mathcal{L}_Zg=\frac{2}{n}(\operatorname{div}Z)g$, with $\operatorname{div}Z\neq\operatorname{const.}$) Then $X=Z+\frac{1}{2}\nabla f$ defines a Ricci almost soliton on $\mathbb{R}\times N(c)$ since
$$
\mathcal{L}_{X}g+\rho=\mathcal{L}_Z g+ H_f+\rho=\left( \frac{2}{n}\operatorname{div}Z+\frac{n-1}{2}c\right) g\,.
$$
We finish the proof showing the non-existence of Ricci almost solitons in any product $N_1^{n_1}(c)\times N_2^{n_2}(-c)$ of two manifolds of  opposite constant sectional curvature. Since $N_1(c)\times N_2(-c)$ is locally symmetric, the Ricci tensor is parallel and 
thus any Ricci almost soliton satisfies
$$
\mathcal{L}_Xg=\lambda g-\rho=\lambda g-\left((n_1-1)c\, g_1\oplus (1-n_2)c\, g_2\right) \,,
$$
which shows that $X$ is a conformal collineation (also called an affine conformal vector field in the literature \cite{D}). Now it follows from \cite[Theorem 9]{Tashiro} that $X$ is a conformal vector field, which is a contradiction since  $N_1(c)\times N_2(-c)$ is not Einstein.
\end{proof}

\begin{remark}
\rm
A vector field $Z$ on a Riemannian manifold $(M,g)$ is said to be an affine conformal vector field if
$$
(\mathcal{L}_Z\nabla)(X,Y)=X(\sigma)Y+Y(\sigma)X-g(X,Y)\nabla\sigma\,,
$$
for some function $\sigma$ on $M$, and $Z$ is said to be an affine vector field if $\mathcal{L}_Z\nabla=0$. In this setting, it follows from \cite{Sharma} that if $(M,g,X,\lambda)$ is a Ricci almost soliton with parallel Ricci tensor, then $X$  is an affine conformal vector field, which becomes an affine vector field if and only if the function $\lambda$ is constant.

Now, work of Tashiro (see \cite{Tashiro}) implies that \emph{a complete proper Ricci almost soliton $(M,g,X,\lambda)$ with parallel Ricci tensor is either Einstein or a product $\mathbb{R}^k\times N$, where $N$ is an Einstein manifold}. Moreover, in the later case the proper Ricci almost soliton decomposes as $X=Z+\xi$ where $(M,g,Z)$ is a Ricci soliton on $\mathbb{R}^k\times N$ and $\xi$ is a conformal vector field as in the proof of Theorem \ref{lema:homras}.
\end{remark}

The next lemma shows that any Einstein Ricci almost soliton decomposes as a sum of a gradient Ricci almost soliton and a conformal vector field, thus extending \cite[Theorem 2.3]{PRRS} to the general non-gradient case.

\begin{theorem}\label{th:Einstein case}
$(M,g)$ is a complete Einstein proper Ricci almost soliton if and only it  is isometric (up to scaling) to:
\begin{enumerate}
\item[(i)] $\mathbb{R}^n$,  $\mathbb{S}^n$ or $\mathbb{H}^n$.
\item[(ii)] $\mathbb{R}\times_{\exp}N$, where $(N,g_N)$ is complete and Ricci flat.
\item[(iii)] $\mathbb{R}\times_{\cosh}N$, where $(N,g_N)$ is a complete Einstein manifold with scalar curvature $\tau_N=-1$.
\end{enumerate}
\end{theorem}

\begin{proof}
If a Ricci almost soliton $\mathcal{L}_Xg+\rho=\lambda g$ is Einstein, then $X$ is a conformal vector field with divergence $\operatorname{div}X=\frac{n}{2}(\lambda-\frac{\tau}{n})$.
Moreover, since $(M,g)$ is Einstein,  
$\varphi=\frac{1}{n}\operatorname{div}X={\frac{1}{2}}(\lambda-\frac{\tau}{n})$ 
is a solution of the Obata's equation
$H_\varphi+\tau\varphi g=0$ (see, for example \cite{Kuhnel-Rademacher97}). If the Ricci almost soliton is proper (i.e., $\lambda\neq\operatorname{const.}$), then $\varphi$ is non-constant and $(M,g)$ is locally a warped product. Moreover it follows from previous work of Kanai \cite[Theorem G]{Kanai83} that in the complete setting $(M,g)$ is either a space of constant curvature or a special class of warped product as in $(ii)$--$(iii)$.

Conversely, if $(M,g)$ is an Einstein manifold admitting a non-zero solution of the Obata's equation
$H_\varphi+k\varphi g=0$, then it is a proper gradient Ricci almost soliton with potential function $f=\varphi$ and $\lambda=\frac{\tau}{n}-k\varphi$, unless $(M,g)$ is the Euclidean space (in which case $f=\varphi$ defines a steady gradient Ricci soliton on $\mathbb{R}^n$).
\end{proof}

\medskip

The non-existence of proper gradient Ricci almost solitons on $\mathbb{R}^n$ is clarified by the following

\begin{lemma}
\label{lemma-irreducibility}
Let $(M,g,f,\lambda)$ be a gradient Ricci almost soliton. If $(M,g)$ admits a non-trivial local de Rham decomposition, then $\lambda$ is constant.
\end{lemma}
\begin{proof}
Let $(M,g,f,\lambda)$ be an $n$-dimensional gradient Ricci almost soliton, i.e.,  $H_f+\rho=\lambda g$.
One has the following immediate consequences of the equation (see for example \cite{Barros-Ribeiro12, PRRS}):
\begin{equation}
\label{eq:gars}
\begin{array}{l}
\Delta f+\tau=n\lambda\,,\\
\noalign{\medskip}
\frac{1}{2}\nabla\tau+\operatorname{Ric}(\nabla f)+\nabla\Delta f=\nabla\lambda\,,\\
\noalign{\medskip}
R(X,Y,\nabla f,Z)=(\nabla_X\rho)(Y,Z)-(\nabla_Y\rho)(X,Z)\\
\noalign{\medskip}
\phantom{R(X,Y,\nabla f,Z)=}
+d\lambda(Y)g(X,Z)-d\lambda(X)g(Y,Z)\,.
\end{array}
\end{equation}
Assume that $(M,g)$ splits as a product $M=M_1\times M_2$, $g=g_1\oplus g_2$ and take vector fields on $M$
of the form $X=(X_1,0)$, $Y=(0,Y_2)$ and $Z=(0,Z_2)$. Then, since both $M_1$ and $M_2$ define parallel distributions on $M$, one has that $R(X,Y,\nabla f,Z)=0$ and 
$(\nabla_X\rho)(Y,Z)=0=(\nabla_Y\rho)(X,Z)$. Hence the third equation in \eqref{eq:gars} gives
$$
0=d\lambda(Y)g(X,Z)-d\lambda(X)g(Y,Z)=-d\lambda(X_1)g_2(Y_2,Z_2),
$$
from where it follows that the soliton function $\lambda$ is constant on $M_1$. A completely analogous argument shows that $\lambda$ is also constant on $M_2$ and $(M,g,f,\lambda)$ becomes a gradient Ricci soliton.
\end{proof}

\medskip

\begin{remark}\rm
\label{re:irreducibility}
Observe that while proper gradient Ricci almost solitons are irreducible, there are non-irreducible examples in the non-gradient case as shown in Theorem \ref{lema:homras}.
\end{remark}

\begin{proof}[Proof of Corollary \ref{th:lochomogars}]
Let $(M,g,f,\lambda)$ be a locally homogeneous proper gradient Ricci almost soliton. It follows from 
Theorem \ref{lema:homras} that $(M,g)$ is locally symmetric and hence Einstein or a product of Einstein spaces. Irreducibility of proper gradient Ricci almost solitons shows that $(M,g)$ is Einstein and thus of non-zero constant sectional curvature since it is locally conformally flat. 
\end{proof}

Let $\rho_1,\dots,\rho_n$ denote the Ricci curvatures (i.e., the eigenvalues of the Ricci operator defined by $g(\operatorname{Ric}X,Y)=\rho(X,Y)$). Clearly any curvature homogeneous Riemannian manifold has constant Ricci curvatures and hence constant scalar curvature.

\begin{theorem}
\label{th:hoy}
Let $(M,g,f,\lambda)$ be a proper gradient Ricci almost soliton. If the Ricci curvatures are constant then  $(M,g)$ is Einstein.
\end{theorem}

\begin{proof}
Since the Ricci curvatures are assumed to be constant, in a suitable orthonormal frame one has
$\operatorname{Ric} = \operatorname{diag}[\rho_1,\dots,\rho_n]$.
Also, by \cite[Remark 2.5]{PRRS}, $d\lambda\wedge df=0$ and thus the soliton function $\lambda$ is a function of the potential function   $f$ in a neighborhood of any point $p\in M$ where $df\neq 0$. Henceforth set $\nabla\lambda = \varphi \nabla f$, where $\varphi=\pm\frac{\|\nabla\lambda\|}{\|\nabla f\|}$ is a smooth function. 

The weighted Laplacian of the scalar curvature of any $n$-dimensional gradient Ricci almost soliton $(M,g,f,\lambda)$ satisfies \cite{PRRS}
\begin{equation}\label{eq:weighted-laplacian}
\frac{1}{2}\Delta_f\tau=\lambda\tau-\|\rho\|^2+(n-1)\Delta\lambda\,.
\end{equation}

Since the scalar curvature $\tau$ of the manifold is constant, the first and second equations in  \eqref{eq:gars} imply that
$\operatorname{Ric}(\nabla f ) = (1-n)\nabla\lambda$ and therefore $\operatorname{Ric}(\nabla f) = (1-n)\varphi\nabla f$, thus showing that $\nabla f$ is an eigenvector of the Ricci operator. By the constancy of the Ricci curvatures one has that $\varphi$ is constant so we may assume without loss of generality that  $(1-n)\varphi=\rho_1$, from where  $\Delta\lambda=\frac{\rho_1}{1-n}\Delta f$. Now,   Equation \eqref{eq:weighted-laplacian} together with Equation \eqref{eq:gars} imply
\begin{equation}
\label{eq1}
0=\frac{1}{2}\Delta_f \tau= \lambda \tau -\|\rho\|^2+(n-1)\Delta \lambda
=\lambda \tau-\|\rho\|^2-\rho_1(n\lambda-\tau)\,.
\end{equation}
If $\tau\neq n\rho_1$, previous expression determines the soliton function as
\[
\lambda=\frac{\|\rho\|^2-\rho_1\tau}{\tau-n\rho_1}\,,
\]
which shows that, in this case,  $\lambda$ is constant and the manifold is a gradient Ricci soliton. 

Next consider the case $\tau=n\rho_1$. Then Equation (\ref{eq1}) becomes
$\|\rho\|^2=n\rho_1^2$. Moreover, since $\tau=\sum_{i=1}^n\rho_i=n\rho_1$, one gets $\rho_n=(n-1)\rho_1-\sum_{i=2}^{n-1}\rho_i$,
and hence
\[
\displaystyle{
n\rho_1^2
= \|\rho\|^2
= \sum_{i=1}^n\rho_i^2
= \sum_{i=1}^{n-1} \rho_i^2 + \left((n-1)\rho_1 - \sum_{i=2}^{n-1}\rho_i\right)^2\,.
}
\]
A straightforward calculation shows  that    this equation is equivalent to
\[
\sum_{i=2}^{n-1} (\rho_i-\rho_1)^2
+ \left(\sum_{i=2}^{n-1} (\rho_i-\rho_1)\right)^2 = 0 \,,
\]
so we conclude that all the Ricci curvatures must be equal. Therefore,  $\operatorname{Ric}=\rho_1\operatorname{Id}$  and   $(M,g)$ is Einstein.
\end{proof}

\medskip

\begin{remark}\label{re:2.5}
\rm
Observe from \eqref{eq1} that a gradient  Ricci almost soliton with constant Ricci curvatures satisfy
$0=\Delta_f \tau = \lambda \tau -\|\rho\|^2-\rho_1(n\lambda-\tau)$ where $\rho_1$ is the (constant) eigenvalue corresponding to the eigenvector $\nabla f$, and hence it is a gradient Ricci soliton if $\tau\neq n\rho_1$. Moreover, if $\tau=n\rho_1$, then 
$\|\rho\|^2=\rho_1\tau =n\rho_1^2$ and it follows from the Schwarz inequality that 
$\|\rho\|^2\ge \frac{\tau^2}{n}=n\rho_1^2$. Hence
$\|\rho\|^2= \frac{\tau^2}{n}$, which provides an alternative proof of the fact that 
any proper gradient Ricci almost soliton with constant Ricci curvatures is Einstein.
\end{remark}

\begin{proof}[Proof of Theorem \ref{th:curvhomog}]
Since $(M,g)$ is curvature homogeneous, the Ricci curvatures are necessarily constant and the
gradient Ricci almost soliton $(M,g,f,\lambda)$ is either a gradient Ricci soliton or $(M,g)$ is Einstein by Theorem \ref{th:hoy}. If $(M,g,f,\lambda)$ is an Einstein gradient Ricci almost soliton, then $\nabla f$ is a conformal vector field. Now it follows from the curvature homogeneity assumption that all the scalar curvature invariants which do not involve covariant derivatives of the curvature tensor (like $\tau$, $\|\rho\|^2$, $\| R\|^2$, $\|W\|^2$) are constant on $M$. Hence $(M,g)$ is locally conformally flat just proceeding as in the proof of Theorem \ref{lema:homras}, and thus of constant sectional curvature. Further the sectional curvature is non-zero by Lemma \ref{lemma-irreducibility}.
\end{proof}

\begin{proof}[Proof of Theorem \ref{th:rigidity}]
Let $(M,g,f,\lambda)$ be a gradient Ricci almost soliton. As in the proof of Theorem  \ref{th:hoy}, set $\nabla\lambda=\varphi\nabla f$ for some function $\varphi$ whenever $\nabla f\neq 0$. If $\rho$ is Codazzi then, proceeding as in \cite{BGR2013}, the third equation in \eqref{eq:gars} gives 
$$
R(X,\nabla f)\nabla f=\varphi g(\nabla f,\nabla f)X-\varphi g(\nabla f,X)\nabla f\,,
$$ 
from where it follows that the radial curvature satisfies $K(X,\nabla f)=-\varphi$
and the scalar curvature $\tau$ is constant. Now, if $(M,g,f,\lambda)$ is radially flat, then $\varphi=0$ and $\lambda$ is constant, thus resulting in a rigid gradient Ricci soliton  \cite[Theorem 1.2]{PW1} proving the Assertion $(i)$.

In order to show $(ii)$, assume $\varphi\neq 0$. Since $\rho$ is Codazzi, we have that 
$$
\frac{\tau}{n}\|\nabla f\|^2=\rho(\nabla f,\nabla f)=-(n-1)\varphi \|\nabla f\|^2
$$ 
from where $\tau=n(1-n)\varphi$, whenever $\nabla f\neq 0$. This shows that $\varphi$ is constant
and taking divergences in the second equation in \eqref{eq:gars} one has
$$
(1-n)\varphi \Delta f=\mathrm{div}(\operatorname{Ric}\nabla f)
=\frac{\tau}{n}\Delta f - \| H_f-\frac{\Delta f}{n} g\|^ 2\,.
$$
Since $\tau=n(1-n)\varphi$ one has that 
$\| H_f-\frac{\Delta f}{n} g\|^2=0$, which shows that the Hessian operator is a multiple of the identity, and hence so is the Ricci operator, which proves that $(M,g)$ is Einstein.
\end{proof}

\begin{proof}[Proof of Theorem \ref{th:2}]
It follows from the assumption in Theorem \ref{th:2}-$(i)$ that, if the Ricci curvatures take values $\rho_{i}\in\{ 0,\lambda\}$, then the scalar curvature is a multiple of $\lambda$, $\tau=k\lambda$, where $k=\operatorname{rank}(\operatorname{Ric})$. Hence
$\lambda\tau-\|\rho\|^2=0$ and thus Equation (\ref{eq:weighted-laplacian}) implies that
$$
\frac{k}{2}\Delta_f\lambda=\frac{1}{2}\Delta_f\tau=(n-1)\Delta\lambda\,,
$$
from where one has $(k+2-2n)\Delta\lambda=kg(\nabla f,\nabla\lambda)$.
Integrating the previous identity on $M$ gives $\int_Mg(\nabla f,\nabla\lambda)=0$, and hence using Equation (\ref{eq:gars}) we get
$$
0=\int_M\lambda\Delta f=\int_M\lambda(n\lambda-\tau)=(n-k)\int_M\lambda^2\,,
$$
from where it follows that $k=n$ and hence $(M,g)$ is Einstein, thus showing Assertion $(i)$.

Next assume that  the scalar curvature is constant and that the Ricci curvatures are  bounded as $0\leq\rho\leq\lambda$. Once again, the weighted Laplacian of the scalar curvature gives (cf. Equation (\ref{eq:weighted-laplacian}))
$$
0=\frac{1}{2}\Delta_f\tau=\lambda\tau-\|\rho\|^2+(n-1)\Delta\lambda\,,
$$
and thus
$$
(n-1)\Delta\lambda=\|\rho\|^2-\lambda\tau=\operatorname{trace}(\operatorname{Ric}\circ(\operatorname{Ric}-\lambda\operatorname{Id}))\leq 0\,.
$$
Now, if the soliton function $\lambda$ is non-negative, it is necessarily constant by the strong maximum principle, provided that it attains a global minimum on $M$, which shows Assertion $(ii)$.
The proof of Assertion $(iii)$ is completely analogous.
\end{proof}

\begin{proof}[Proof of Remark \ref{re.nuevo}]
Let $(M,g,X,\lambda)$ be an $n$-dimensional compact shrinking or expanding Ricci almost soliton whose Ricci curvatures take values $\rho_i\in\{0,\lambda\}$. Proceeding as in the proof of Theorem \ref{th:2}, the scalar curvature $\tau$ is a multiple of the soliton function $\tau=\lambda\operatorname{rank}(\operatorname{Ric})$.
Tracing in \eqref{eq:RicciSoliton} one has 
$\tau=\lambda\operatorname{rank}(\operatorname{Ric})=n\lambda-2\operatorname{div}X$, and hence
$$
0=2\int_M\operatorname{div}X=(n-\operatorname{rank}(\operatorname{Ric}))\int_M\lambda
$$
from where it follows that $(M,g)$ is Einstein, and thus a Euclidean sphere \cite[Corollary 1]{BBR-2014}.

Assertion $(ii)$ in Theorem \ref{th:2} are also valid when the generic case of non-gradient Ricci almost solitons is considered. Take  the $X$-Laplacian $\Delta_X\psi=\Delta\psi-\langle X,\nabla\psi\rangle$ to compute (see, for example, \cite{BBR-2014})
$$
0=\frac{1}{2}\Delta_X\tau=\lambda\tau-\|\rho\|^2+(n-1)\Delta\lambda\,,
$$
from where $(n-1)\Delta\lambda=\|\rho\|^2-\lambda\tau=\operatorname{trace}(\operatorname{Ric}\circ(\operatorname{Ric}-\lambda\operatorname{Id}))\leq 0$. Then the constancy of the soliton function follows as an application of the maximum principle. Assertion $(iii)$ is also valid in the generic case, just proceeding in an analogous way.
\end{proof}

\medskip

\begin{remark}\rm
\label{re:final}
Let $(M,g,f,\lambda)$ be a gradient  Ricci almost soliton with constant scalar curvature $\tau$. Equation \eqref{eq:gars} implies  that 
$\operatorname{Ric}(\nabla f)=(1-n)\nabla\lambda$. Since $\lambda$ is a function of $f$, one has that $\nabla\lambda\in\operatorname{span}\{\nabla f\}$, and thus $\nabla f$ is an eigenvector of the Ricci operator. Hence, $\nabla f$ is also an eigenvector of the Hessian operator, which shows that \emph{the integral curves of $\nabla f$ are unparametrized geodesics in the constant scalar curvature setting}.
\end{remark}


\begin{thebibliography}{99}
\bibitem{BBR-2014}
A. Barros, R. Batista, and E. Ribeiro Jr.,
Compact almost Ricci solitons with constant scalar curvature are gradient,
\emph{Monatsh. Math.} \textbf{174} (2014), 29--39.

\bibitem{BBR--2012}
A. Barros, R. Batista, and E. Ribeiro Jr.,
Rigidity of gradient almost Ricci solitons,
\emph{Illinois J. Math.} \textbf{56} (2012), 1267--1279.

\bibitem{Barros-Ribeiro12} 
A. Barros and E. Ribeiro Jr., 
Some characterizations for compact almost Ricci solitons, 
\emph{Proc. Amer. Math. Soc.} \textbf{140} (2012), 1033--1040.

\bibitem{BGR2013}
A. Barros, J. N. Gomes, and E. Ribeiro Jr., 
A note on rigidity of the almost Ricci soliton,
\emph{Arch. Math. (Basel)} \textbf{100} (2013), 481--490. 

\bibitem{BKoVa}
E. Boeckx, O. Kowalski, and L. Vanhecke, 
\emph{Riemannian manifolds of conullity two}, World
Scientific Publishing Co., Inc., River Edge, NJ, 1996.

\bibitem{BCR-2013}
A. Brasil, E. Costa, and E. Ribeiro Jr.,
Hitchin-thorpe inequality and Kaehler metrics for compact almost Ricci soliton,
\emph{Ann. Mat. Pura Appl. (4)} \textbf{193} (2014), 1851--1860.

\bibitem{CMMR}
G. Catino, P. Mastrolia, D. D. Monticelli, and M. Rigoli,
On the geometry of gradient Einstein-type manifolds,
arXiv:1402.3453 [math.DG].

\bibitem{D}
K. L. Duggal, 
Affine conformal vector fields in semi-Riemannian manifolds,
\emph{Acta Appl. Math.} \textbf{23} (1991), 275--294.

\bibitem{FL-GR}
M. Fern\'andez-L\'opez and E. Garc\'\i a-R\'\i o,
On gradient Ricci solitons with constant scalar curvature,
\emph{Proc. Amer. Math. Soc.}, to appear.



\bibitem{Gh}
A. Ghosh,
Certain Contact Metrics as Ricci Almost Solitons,
\emph{Result. Math.} \textbf{65} (2013), 81--94.


\bibitem{Gi}
P. Gilkey,
\emph{The geometry of curvature homogeneous pseudo-Riemannian manifolds},
ICP Advanced Texts in Math. \textbf{2}, Imperial College Press, London, 2007.

\bibitem{Gray}
A. Gray,
Einstein-like manifolds which are not Einstein,
\emph{Geom. Dedicata} \textbf{7} (1978), 259--280.


\bibitem{Jablonski}
M. Jablonski,
Homogeneous Ricci solitons are algebraic,
\emph{Geometry \& Topology} \textbf{18} (2014), 2477--2486.

\bibitem{Kanai83} M. Kanai, 
On a differential equation characterizing the Riemannian structure of a manifold, 
\emph{Tokyo J. Math.} \textbf{6} (1983), 143--151.
    

\bibitem{Kuhnel-Rademacher97} 
W. K\"{u}hnel and H.B. Rademacher, 
Conformal vector fields on pseudo-Riemannian spaces, 
\emph{Differential Geom. Appl.} \textbf{7} (1997), 237--250.
    
\bibitem{Lauret}
J. Lauret, 
Ricci soliton homogeneous nilmanifolds,
\emph{Math. Ann.} \textbf{319} (2001), 715--733.
    

\bibitem{Ma}
G. Maschler,
Almost soliton duality,
arXiv:1301.0290 [math.DG].
   

\bibitem{Petersen-Wylie09} 
P. Petersen and W. Wylie, 
On Ricci solitons with symmetries, 
\emph{Proc. Amer. Math. Soc.} \textbf{1377} (2009), 2085--2092.

\bibitem{PW1} 
P. Petersen and W. Wylie, 
Rigidity of gradient Ricci solitons, 
\emph{Pacific J. Math.} \textbf{241} (2009), 329--345.

\bibitem{PRRS}
S. Pigola, M. Rigoli, M. Rimoldi, and A. Setti, 
Ricci almost solitons,
\emph{Ann. Sc. Norm. Super. Pisa Cl. Sci. (5)} \textbf{10} (2011), 757--799.

\bibitem{Sharma}
R. Sharma,
Almost Ricci solitons and K-contact geometry, 
\emph{Monatsh. Math.} \textbf{175} (2014), 621--628.


\bibitem{Spiro}
A. Spiro, 
A remark on locally homogeneous Riemannian spaces, 
\emph{Result. Math.} {\bf 24} (1993), 318--325.

\bibitem{Takagi}
H. Takagi,
Conformally flat Riemannian manifolds admitting a transitive group of isometries,
\emph{Tohoku Math. J.} \textbf{27} (1975), 103--110.


\bibitem{Tashiro}
Y. Tashiro, 
On conformal collineations,
\emph{Math. J. Okayama Univ.} \textbf{10} (1960), 75--85.

\bibitem{WGX}
Q. Wang, J. N. Gomes, C. Xia,
H-almost ricci soliton,
 	arXiv:1411.6416 [math.DG]
\end{thebibliography}
\end{document}